\documentclass[final,3p]{elsarticle}

\usepackage{lineno}
\usepackage[colorlinks=true,breaklinks=true,pdftex]{hyperref}
\modulolinenumbers[1]

\usepackage[english]{babel}
\usepackage[utf8]{inputenc}
\usepackage{typearea}
\usepackage{latexsym}
\usepackage{lineno,hyperref}
\usepackage{amsmath,amssymb,mathtools,amsthm}
\usepackage{algorithm}
\usepackage{algorithmic}
\usepackage{graphicx}
\usepackage[abs]{overpic}
\usepackage{nicefrac}

\newcommand{\M}{\mathcal{M}}
\newcommand{\B}{\mathcal{B}}

\newcommand*{\quark}{\setbox0\hbox{$x$}\hbox to\wd0{\hss$\cdot$\hss}}

\usepackage{xcolor}
\newcommand{\todo}[1]{\bgroup\color{red}#1\egroup}

\newtheorem{theorem}{Theorem}

\newtheorem{definition}{Definition}[section]

\journal{GMP 2023}

\biboptions{authoryear}

\let\cite\citep

\begin{document}

\begin{frontmatter}

\title{
Sasaki Metric for Spline Models\\of Manifold-Valued Trajectories}


\author[first]{Esfandiar Nava-Yazdani\corref{cor}}
\cortext[cor]{Corresponding author}
\ead{navayazdani@zib.de}
\author[first,last]{Felix Ambellan}
\ead{ambellan@zib.de}
\author[first,last]{Martin Hanik}
\ead{hanik@zib.de}
\author[last]{Christoph von Tycowicz}
\ead{vontycowicz@zib.de}
\address[first]{Zuse Institute Berlin, Germany}
\address[last]{Freie Universität Berlin, Germany}

\begin{abstract}
We propose a generic spatiotemporal framework to analyze manifold-valued measurements, which allows for employing an intrinsic and computationally efficient Riemannian hierarchical model. Particularly, utilizing regression, we represent discrete trajectories in a Riemannian manifold by composite B\' ezier splines, propose a natural metric induced by the Sasaki metric to compare the trajectories, and estimate average trajectories as group-wise trends. We evaluate our framework in comparison to state-of-the-art methods within qualitative and quantitative experiments on hurricane tracks. Notably, our results demonstrate the superiority of spline-based approaches for an intensity classification of the tracks.
\end{abstract}


\end{frontmatter}


\section{Introduction}
A wide range of applications in many areas including  morphology, action recognition and computer vision requires longitudinal study of manifold-valued data. Approaches relying on the Euclidean structure of the embedding space or local coordinates do not correctly incorporate the inherent underlying geometry. Particularly, representing the geometry of trajectories in Euclidean spaces undermines the ability to represent natural variability in populations. The need for geometry-aware approaches has led to an increasing number of works on comparison of manifold-valued trajectories via Riemannian metrics~\cite{MuFl2012, le2015reparameterization, schiratti2015learning, bauer2017numerical,debavelaere2020learning, RePEc:biomet2021, overviewintrinsic, shao2022intrinsic, NavaYazdaniHegevonTycowicz2022, HanikHegevonTycowicz2022}. 

A Riemannian framework for statistics and analysis in the space of trajectories is appealing because it provides a rich structure for which powerful tools are already available. 
Since the space of general trajectories is infinite-dimensional, it is highly complex. Particularly for numerical treatments, it is thus desirable to replace it with an adequate finite dimensional space. A natural and powerful method towards this goal is to represent the trajectories by finitely parametrized curves via regression, which is a ubiquitous tool in various application fields. This includes piece-wise geodesic, kernel, polynomial, logistic (see~\cite{regoverview} and the references therein), and Fourier regression~\cite{Suparti_2019}. 

The most common method for regression on a manifold is geodesic regression (cf.~\cite{SM2006,niethammer2011geodesic,Fletcher2013,NavaYazdaniHegeSullivanetal.2020,NavaYazdaniHegevonTycowicz2022}), which is the counterpart of linear regression in Euclidean space. Often, the relation between the variables is highly complex, and 
therefore geodesic regression can prove to be inadequate for model selection toward a best-fitting curve. In such cases, it is necessary to resort to a more flexible model for regression. For this purpose, a generalization that relies on intrinsic polynomials in Riemannian manifolds was proposed 
in~\cite{Hinkle2014}. However, due to their flexibility and effective evaluation in terms of simple iterative algorithms for a constructive approach, B\'ezier splines~\cite{PN2007,GousenbourgerMassartAbsil2019} are more adequate as best-fitting curves for
regression. Note that, inconsistencies are minimized by considering best-fitting smooth curves.

In this work, we employ the spline regression approach that was recently proposed in~\cite{Hanik_ea2020} and~\cite{HanikHegevonTycowicz2022} to approximate trajectories. We show that the resulting space of B\'ezier splines is finite-dimensional and that it can be endowed with a natural Riemannian metric that stems from the famous Sasaki metric.   
Utilizing this structure, one can intrinsically employ principal component analysis and, particularly, compute group trends as average trajectories; the latter establishes a generic hierarchical model~(cf.~\cite{MuFl2012,NavaYazdaniHegevonTycowicz2022,HanikHegevonTycowicz2022}).

We verify the virtue of the proposed framework by applying it to data on the 2010-2021 Atlantic hurricane tracks. Also referred to as tropical cyclones, hurricanes belong to the most destructive natural disasters. They can have an enormous impact on environment, economy, and human life~\cite{weinkle2018normalized}.
We compare our approach to state-of-the-art methods regarding an intensity classification for the tracks. This comparison substantiates the advantages of the proposed approach. The code implementing our method is publicly available in \texttt{morphomatics}~\cite{Morphomatics} v3.0 and the experiments in \url{https://github.com/morphomatics/GeometricHurricaneAnalysis}.
This paper is organized as follows. In the next section, we describe the mathematical framework and proposed hierarchical model. Therein, we present the spline regression and the Sasaki metric, which will be used to compare the splines representing the trajectories. Section~\ref{sec:hur} presents the application of our approach to hurricane tracks. Therein, we present results for the intrinsic hierarchical analysis of the hurricane tracks as well as our classification experiment and a discussion of the numerical outcome.
\section{Mathematical Framework and Method}
\label{sec:method}
\subsection{Background}
We start by summarizing important facts on the geometry of the tangent bundle. For more background information see, e.g., \cite{doCarmo1992}. 
Let $(M, g)$ be a $n$-dimensional smooth\footnote{In this work, ``smooth'' stands for ``infinitely differentiable.''} Riemannian manifold and $TM$ its tangent bundle with bundle projection $\pi: TM \to M$, $(p,v) \mapsto p$. We denote the tangent space at $p\in M$ by $T_pM$ and the restriction of $g$ to $T_pM$ by $g_p$. We always assume that maps are smooth. 

Being a $2n$-dimensional manifold itself, the natural Riemannian metric of $TM$ is the Sasaki metric. 
For the definition of the latter, recall that the tangent bundle $TTM$ of $TM$ is the direct sum of a vertical subbundle $VTM$ (the kernel of the derivative $\textnormal{d}\pi$ of $\pi$) and a horizontal subbundle $HTM$ determined by the Levi-Civita connection $\nabla$ of $M$ as vectors tangent to parallel vector fields
, and both have rank $n$. Intuitively, horizontal vectors are directions in which only the footpoint changes, whereas the latter is constant in vertical directions.  
The Sasaki metric~\cite{sasaki1962differential} is the unique Riemannian metric on $TM$ with the following properties:
\begin{itemize}
    \item The bundle projection $\pi$ is a Riemannian submersion, i.e.\ $\pi$ has maximal rank and $\textnormal{d}\pi$ preserves lengths of horizontal vectors.
    \item For any $p \in M$, the restriction of $\widetilde{g}$ to the tangent space $T_pM \subset TM$ coincides with $g_p$.
    \item Let $u$ be a \textit{parallel} vector field along a curve $q: (-\varepsilon,\varepsilon) \to M$. Define $\zeta = (-\varepsilon,\varepsilon) \to TM$, $t \mapsto (q(t), u(t))$. Let further $p=q(0)$ lie on $q$ and $v: (-\varepsilon,\varepsilon) \to T_pM$ such that $u(0) = v(0)$. Define $\eta: (-\varepsilon,\varepsilon) \to TM$, $t \mapsto (p, v(t))$. Then $\dot{\zeta}(0) := \frac{d}{dt}\zeta(0) \in H_{(p,u(0))}TM$ and $\dot{\eta}(0) = \frac{d}{dt}\eta(0) \in V_{(p,u(0))}TM$ are orthogonal.
\end{itemize}
    It follows that horizontal and vertical vectors are orthogonal.
    
    Let $(p,u) \in TM$. Because $H_{(p,u)}TM$ and $V_{(p,u)}TM$ are $n$-dimensional vector spaces, both can be identified with $T_pM$; in other words, we can view an element of $T_{(p,u)}TM$ as a tuple $(v,w) \in (T_pM)^2$. As horizontal and vertical components are orthogonal, the Sasaki metric between $(v_1,w_1), (v_2,w_2) \in T_{(p,u)}TM$ then reads
    \begin{equation} \label{def:Sasaki_metric}
        \widetilde{g}_{(p,u)}\big((v_1,w_1), (v_2,w_2)\big) = g_p(v_1,v_2) + g_p(w_1,w_2).
    \end{equation}

    Geodesics of the Sasaki metric can be characterized in terms of geometric features of the underlying space.
    Let $\eta=(q,u): (-\varepsilon,\varepsilon) \to TM$ be a curve in $TM$ and $\dot{\eta} = (v,w)$. Denoting the Riemannian curvature tensor of $M$ by $R$, $\eta$ is a geodesic if (and only if) the coupled system
    \begin{align*}
        \nabla_vv &= -R(u, w)v, \\
        \nabla_vw &= 0
    \end{align*}
    holds. Intuitively, the above equations say that the footpoint curve $q$ bends according to the curvature of $M$, while the vector component $u$ changes at a constant rate.
    
    Further properties of $(TM, \widetilde{g})$, e.g.\ its curvature tensor and Levi-Civita connection, can be found in~\cite{gudmundsson2002geometry}.

\subsection{B\' ezier Splines}

Let $U$ denote a normal convex (sometimes also called totally normal) neighborhood in $M$. Thus, the diameter of $U$ is less than the injectivity radius of $M$ and the Riemannian logarithm $\log$ is defined on the whole of $U$. In particular, for any $p,q\in U$ there is a unique geodesic from $p$ to $q$ that never leaves $U$ given by 
\[
\gamma(t; p,q):=\exp_p(t\log_pq), \quad 0\leq t\leq 1\,.
\]
Requiring that the data lies in $U$ ensures well-posedness of our construction. If $M$ is a Hadamard manifold, then we can simply set $U=M$.
For points $p_0,\dots,p_k \in U$, we set the following.
\begin{definition}
[De Casteljau's Algorithm on Manifolds]\label{dc}
    \begin{align*}
	&\beta_i^0(t):=p_i, \quad i=0,\dots,k-r,\\
        &\beta_i^r(t):=\gamma(t; \beta_i^{r-1}(t), \beta_{i+1}^{r-1}(t)), \quad r=1,\dots,k, \quad 0\leq t\leq 1.
    \end{align*}
\end{definition}
We call $\beta := \beta^k_0$ \textit{Bézier curve of degree $k$ with control points} $p_0,\dots,p_k$.  
For more details on them and some applications, we refer to \cite{PN2007} and \cite{NP2013}. Notation-wise, whenever we want to make the dependency on the control points clear, we write $\beta(t; p_0,\dots,p_k)$ instead of $\beta(t)$ for the value of $\beta$ at $t$.

This work is centered around cubic Bézier curves, i.e. we consider the case $k=3$.
Then, control points and velocities at endpoints satisfy
\begin{align} \label{eq:point_from_vec}
\begin{split}
  p_1&=\exp_{p_0} \left( \frac{1}{3}\dot{\beta}(0) \right),\\
p_{2}&=\exp_{p_3} \left(-\frac{1}{3}\dot{\beta}(1) \right).  
\end{split}
\end{align}
or equivalently
\begin{align} \label{eq:vec_from_point}
\begin{split}
\dot{\beta}(0)&=3\log_{p_0}p_1,\\
\dot{\beta}(1)&=-3\log_{p_3}p_2. 
\end{split}
\end{align}

These properties can be exploited to define differentiable (or composite) B\'ezier splines~\cite{GousenbourgerMassartAbsil2019, Hanik_ea2020}. In the following, we explain this construction. For $i=0,\dots,L-1$, let $\big( p^{(i)}_0, p^{(i)}_1, p^{(i)}_2, p^{(i)}_{3} \big)$ be the control points of $L\ge2$ cubic Bézier curves $\beta^{(0)},\dots,\beta^{(L-1)}$ such that 
\begin{equation} \label{eq:C1_condition}
    p^{(i)}_{3} = p^{(i+1)}_0 \quad \text{and} \quad  \gamma \left( 2;p^{(i)}_{2},p^{(i)}_{3} \right) = p^{(i+1)}_1
\end{equation}
for all $i=0,\dots,L-2$.
The \textit{cubic Bézier spline $B$ with control points} 
$$\left( p^{(i)}_0,p^{(i)}_1,p^{(i)}_2,p^{(i)}_3 \right)_{i=0,\dots,L-1}$$ 
is then defined by
\begin{equation*} \label{eq: spline}
    B(t) := \begin{cases} \beta^{(0)}\left(t;p^{(0)}_0,p^{(0)}_1,p^{(0)}_2,p^{(0)}_3\right), &\quad t \in [0,1],\\ \beta^{(i)} \left(t-i;p^{(i)}_0,p^{(i)}_1,p^{(i)}_2,p^{(i)}_3 \right), &\quad t \in (i,i+1], \quad i=1,\dots,L-1. \end{cases}
\end{equation*}
Note that it is $C^1$ by construction.
We also consider Bézier curves as splines with a single segment (the $L=1$ case). 

\subsection{The B\' ezierfold}
In the following, we present a natural Riemannian metric for the splines. To this end, we introduce the following.
\begin{definition}[Bézierfold of cubic splines]\label{def:bezierfold}
Let $U$ be a normal convex neighborhood of an $n$-dimensional Riemannian manifold $M$. We define the \textnormal{Bézierfold} $\mathcal{B}^L_3(U)$ of cubic splines by
\begin{align*}
\mathcal{B}^L_3(U) := \{B: [0,L]\to U\ \big|\ B \text{ is a cubic Bézier spline}\text{ with $L$ segments}\}.
\end{align*}
\end{definition}
Now, fix $L$ and $U$. To simplify notation, we write $\mathcal{B}$ for $\mathcal{B}^L_3(U)$. As remarked (but not proven) in~\cite{HanikHegevonTycowicz2022}, $\mathcal{B}$ can be given the structure of a smooth manifold when splines are identified with a suitable subset of their control points.  
In the following, we will rigorously prove this assertion, albeit using a different identification map. The advantage of the latter is that it immediately allows us to define our novel metric on $\mathcal{B}$.
\begin{theorem}
    The Bézierfold $\mathcal{B}$ can be given the structure of a smooth $(2L+2)n$-dimensional manifold.
\end{theorem}
\begin{proof}
    We define the map
    \begin{align} \label{def:F}
    \begin{split}
        F: \mathcal{B} &\to (TU)^{L+1}, \\
        B &\mapsto \bigg( \big(B(0), \dot{B}(0) /3 \big), \dots, \big( B(L), \dot{B}(L) / 3 \big) \bigg).
        \end{split}
    \end{align}
    Equation (\ref{eq:vec_from_point}) yields
    \begin{align*}
        F(B) = \Bigg( \bigg(p^{(0)}_0, \log_{p^{(0)}_0}\big(p^{(0)}_1\big) \bigg), \dots, \bigg(p^{(L-1)}_0, \log_{p^{(L-1)}_0}&\big(p^{(L-1)}_1\big) \bigg), \\
        &\bigg(p^{(L-1)}_3, -\log_{p^{(L-1)}_3}\big(p^{(L-1)}_2\big) \bigg) \Bigg).
    \end{align*}
    Crucially, $F$ is bijective: It follows from (\ref{eq:point_from_vec}) and (\ref{eq:C1_condition}) that $F^{-1}$ is the map that assigns to each element 
    $$\big((p_0, u_0), (p_1, u_1) \dots, (p_L, u_L)\big) \in (TU)^{L+1}$$
    the Bézier spline with control points 
    \begin{align} \label{eq:F-1}
        \left( p^{(i)}_0,p^{(i)}_1,p^{(i)}_2,p^{(i)}_3 \right) &= \left( p_i, \exp_{p_i}(u_i), \exp_{p_{i+1}}(-u_{i+1}), p_{i+1} \right), \quad i=0,\dots,L-1.
    \end{align}
    The manifold structure is now obtained by pulling back \cite[Ch.\ 30 § 9]{Postnikov2013} the (product) structure of $(TU)^{L+1}$ along $F$. We thus obtain the induced topology. Furthermore, if $V\subseteq(TU)^{L+1}$ and $\phi:V \to \mathbb{R}^{(2L+2)n}$ form a chart $(V,\phi)$ of $(TU)^{L+1}$, then $(F^{-1}(V), \phi \circ F)$ is a chart of $\mathcal{B}$. The set of all charts that are constructed in this way constitutes the maximal (smooth) atlas of $\mathcal{B}$.
\end{proof}
Note that, with the above construction, $F$ is a diffeomorphism. In Fig.~\ref{fig:spline_encoding}, we visualize a Bézier spline with 2 segments on $S^2$ together with the 3 elements of $TS^2$ to which it is mapped by $F$.

\begin{figure}[t]
	\centering
        \includegraphics[width=0.5\textwidth]{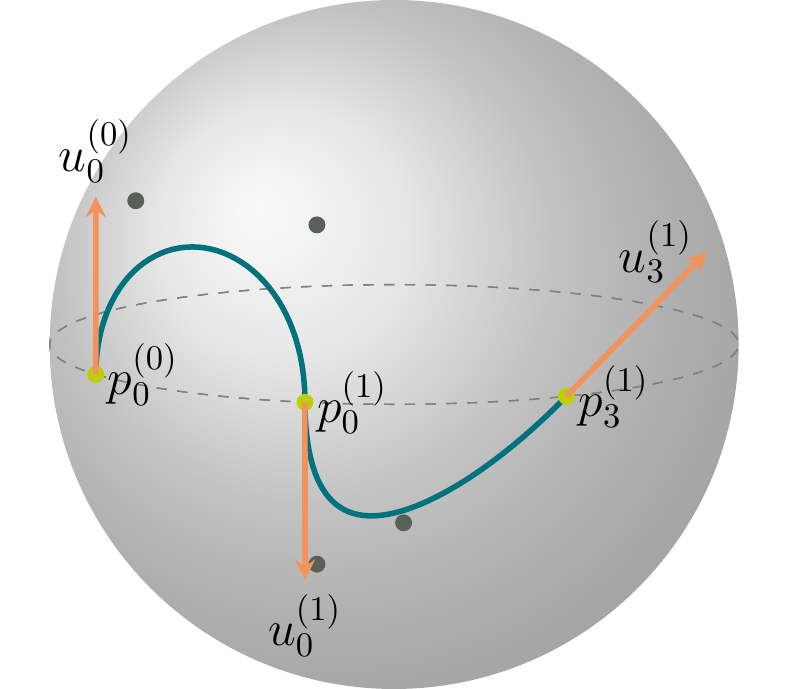}
	\caption{Cubic Bézier spline with two segments on the 2-dimensional sphere $S^2$. Footpoints are shown in green while vectors are in orange. Here, $u_0^{(i)}:=\textnormal{log}_{p_0^{(0)}}(p_1^{(0)})$, $u_0^{(1)}:=\textnormal{log}_{p_0^{(1)}}(p_1^{(1)})$, and $u_3^{(1)}:=-\textnormal{log}_{p_3^{(1)}}(p_2^{(1)})$. Control points that are not used in the representation are gray.}
	\label{fig:spline_encoding}
\end{figure}

\subsection{Sasaki Metric for the Bézierfold}
We can now use $F$ to pull the (product) Sasaki metric back to the B\'ezierfold. To this end, we first investigate how the $F^{-1}$ characterizes the tangent spaces of the Bézierfold. Therefore, we use the following notation. If $f: M \to N$ is a map between manifolds $M$ and $N$, we denote the derivative of $f$ at $p$ in direction $u \in T_pM$ by $\textnormal{d}_pf(u)$. When we differentiate the exponential $\exp_p(u)$ w.r.t.\ the footpoint at $p$, we write $(\textnormal{d}_p\exp_{(\cdot)}(u))(v)$.

Let $B \in \mathcal{B}$. Because of (\ref{eq:F-1}), there is $((p_0, u_0), (p_1, u_1) \dots, (p_L, u_L)) \in (TU)^{L+1}$ such that the control points of $B$ are given by 
$$\left( p_i, \exp_{p_i}(u_i), \exp_{p_{i+1}}(-u_{i+1}), p_{i+1} \right)_{i=0,\dots,L-1}.$$
Since derivatives of diffeomorphisms are everywhere isomorphisms between tangent spaces, we find that $T_B\mathcal{B}$ is the image of the tangent space $T_{((p_0, u_0),\dots, (p_L, u_L))}(TU)^{L+1}$ under $\textnormal{d}F^{-1}$. In the product structure, we can compute this image component-wise; therefore
\begin{align} \label{eq:tangent_space}
\begin{split}
T_B\mathcal{B} = \{ X: [0,L] \to TM\ \big|\  \exists \ &(v_0, w_0) \in T_{(p_0, u_0)}TM, \dots, \\ 
&(v_L, w_L) \in T_{(p_L, u_L)}TM: X = \sum_{i=0}^L \textnormal{d}_{(p_i,u_i)}F^{-1}\big((v_i,w_i)\big)\}.
\end{split}
\end{align}
Every element $X \in T_B\mathcal{B}$ is thus a vector field along $B$. We say that it is induced by the vectors $(v_0, w_0), \dots, (v_L, w_L) \in (TM)^2$.

We now use $B_j^{(i)}$ when we view only the $j$-th control point of the $i$-th segment of $B$ as a variable, while all others (as well as $t$) are fixed.
Definition (\ref{def:F}) then implies that, for each $i \in \{0,\dots, L\}$, the derivative 
of $F^{-1}$ at $(p_i,u_i)$ in direction $(v_i,w_i)$ is $$\textnormal{d}_{(p_i,u_i)}F^{-1}\big((v_i,w_i)\big) = \textnormal{d}_{(p_i,u_i)}B\big((v_i,w_i)\big).$$
For $i \in \{1,\dots, L-1\}$ the chain rule then yields
\begin{align} \label{eq:derivative_B_1}
\begin{split}
    \textnormal{d}_{(p_i,u_i)}B\big((v_i,w_i)\big) = & \ \textnormal{d}_{\exp_{p_{i}}(-u_i)}B^{(i-1)}_2 \big(\textnormal{d}_{-u_i} \exp_{p_i}(w_i) \big) \\
    &+ \textnormal{d}_{\exp_{p_{i}}(-u_i)}B^{(i-1)}_2 \big(\textnormal{d}_{p_i} \exp_{(\cdot)}(-u_i)(v_i) \big) \\ 
    &+ \textnormal{d}_{\exp_{p_{i}}(u_i)}B^{(i)}_1 \big(\textnormal{d}_{u_i}\exp_{p_i}(w_i) \big) \\
    &+ \textnormal{d}_{\exp_{p_{i}}(u_i)}B^{(i)}_1\big(\textnormal{d}_{p_i} \exp_{(\cdot)}(u_i)(v_i) \big) + \textnormal{d}_{p_i}B^{(i)}_0(v_i).
\end{split}
\end{align}
For $i=0$ we get
\begin{align} \label{eq:derivative_B_2}
\begin{split}
    \textnormal{d}_{(p_0,u_0)}B\big((v_0,w_0)\big) =& \ \textnormal{d}_{\exp_{p_0}(u_0)}B^{(0)}_1 \big(\textnormal{d}_{u_0}\exp_{p_0}(w_0) \big) \\
    &+ \textnormal{d}_{\exp_{p_{0}}(u_0)}B^{(0)}_1\big(\textnormal{d}_{p_0} \exp_{(\cdot)}(u_0)(v_0) \big) + \textnormal{d}_{p_0}B^{(0)}_0(v_0),
\end{split}
\end{align}
and for $i = L$
\begin{align} \label{eq:derivative_B_3}
\begin{split}
    \textnormal{d}_{(p_L,u_L)}B\big((v_L,w_L)\big) = & \ \textnormal{d}_{\exp_{p_{L}}(-u_L)}B^{(L-1)}_2 \big(\textnormal{d}_{-u_L} \exp_{p_L}(w_L) \big) \\
    &+ \textnormal{d}_{\exp_{p_{L}}(-u_L)}B^{(L-1)}_2 \big(\textnormal{d}_{p_L} \exp_{(\cdot)}(-u_L)(v_L) \big) + \textnormal{d}_{p_L}B^{(L-1)}_3(v_L).
\end{split}
\end{align}

Note that we only take derivatives of single Bézier curves (the $(i-1)$-th and $i$-th segment) w.r.t.\ their control points. They are given by ``concatenated'' Jacobi fields~\cite[Thm.\ 7]{BergmanGousenbourger2018}. Furthermore, the derivatives of the exponential map are also Jacobi fields~\cite[Sec.\ 3.1]{Fletcher2013}. If $M$ is a symmetric space, then (\ref{eq:derivative_B_1}), (\ref{eq:derivative_B_2}), (\ref{eq:derivative_B_3}), and thus elements of the tangent space (\ref{eq:tangent_space}) can be calculated explicitly~\cite[p.\ 121]{doCarmo1992}. 

We endow $T\mathcal{B}$ with the pullback $g^*$ of the product Sasaki metric under $F$.
To this end, let $X, Y \in T_B\mathcal{B}$ and $(v^X_0, w^X_0), \dots, (v^X_L, w^X_L) \in (TM)^2$ and $(v^Y_0, w^Y_0), \dots, (v^Y_L, w^Y_L) \in (TM)^2$ be the vectors that induce them.
Then, using (\ref{def:Sasaki_metric}), we obtain
\begin{align*}
g^*_B(X, Y) &:= \sum_{i=0}^L\widetilde{g}_{(p_i,u_i)}\big((v^X_i,w^X_i), (v^Y_i, w^Y_i)\big) \\
&= \sum_{i=0}^Lg_{p_i}(v^X_i,v^Y_i) + g_{p_i}(w^X_i, w^Y_i).
\end{align*}

Importantly, as $\mathcal{B}$ and $(TM)^{L+1}$ are now isometric (with isometry $F$), we can always use the latter for computations, only transforming the results to $\mathcal{B}$ (with $F^{-1}$) as the final step. In particular, for the calculations in the next section, it is never necessary to evaluate a vector field along a Bézier curve explicitly; one only needs the vectors that induce it. This is one of the major advantages of our construction.
\section{Application: Hurricane Tracks}
\label{sec:hur}
Tropical cyclones, also referred to as hurricanes or typhoons, belong to the most supreme natural phenomena with enormous impact on environment, economy, and human life. The most common indicator for the intensity of a hurricane is its maximum sustained wind (maxwind), which classifies the storm into categories via the Saffir–Simpson hurricane wind scale. For instance, maxwind $\geq 137$ knots corresponds to category 5. The maximal category over a track is called its category. Thus, the same applies to the maxwind. 

High variability of tracks and out-most complexity of hurricanes has led to a huge number of works to classify, rationalize and predict them. We refer to~\cite{REKABDARKOLAEE2019351} for a Bayesian function model,~\cite{Asif2018PHURIEHI} for intensity estimation via machine learning and the overview of recent progress in tropical cyclone intensity forecasting~\cite{RecentProgress}, and~\cite{Snaiki2020RevisitingHT}. We remark that many approaches are not intrinsic and use linear approximations. 
Notable exceptions are the works \citep{SriTrjHur, bauerHomog} that employ an intrinsic Riemannian approach based on the squared root velocity framework, albeit only as illustrative examples and without consideration of maxwind.

\subsection{Dataset}
We verify the effectiveness of the proposed framework by applying it to the 2010-2021 Atlantic hurricane tracks (total number 218) from the HURDAT 2 database provided by the U.S. National Oceanic and Atmospheric Administration publicly available on \url{https://www.nhc.noaa.gov/data/}. The data under consideration comprises measurements of latitude, longitude, and maxwind on a 6 hours base. Tracks are represented as discrete trajectories in $S^2$. The number of points constituting a track varies from 13 to 96 (on average 32). Fig.~\ref{fig:hurricanes} illustrates this data set with a histogram of maximum sustained winds and a visualization of the 2010 hurricane tracks.

\begin{figure}[tb]
	\centering
  \includegraphics[width=0.43\textwidth,trim=0 8 0 0,clip]{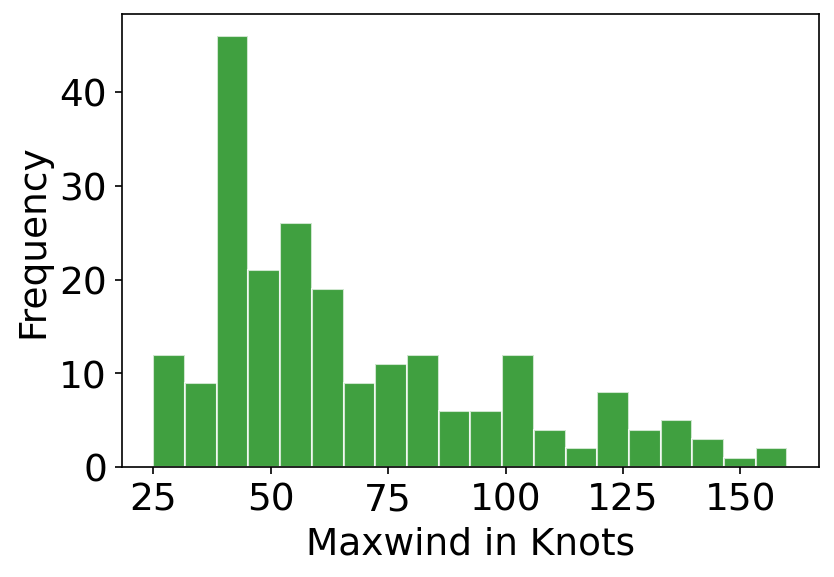}
  \includegraphics[width=0.5\textwidth,trim=0 0 0 50,clip]{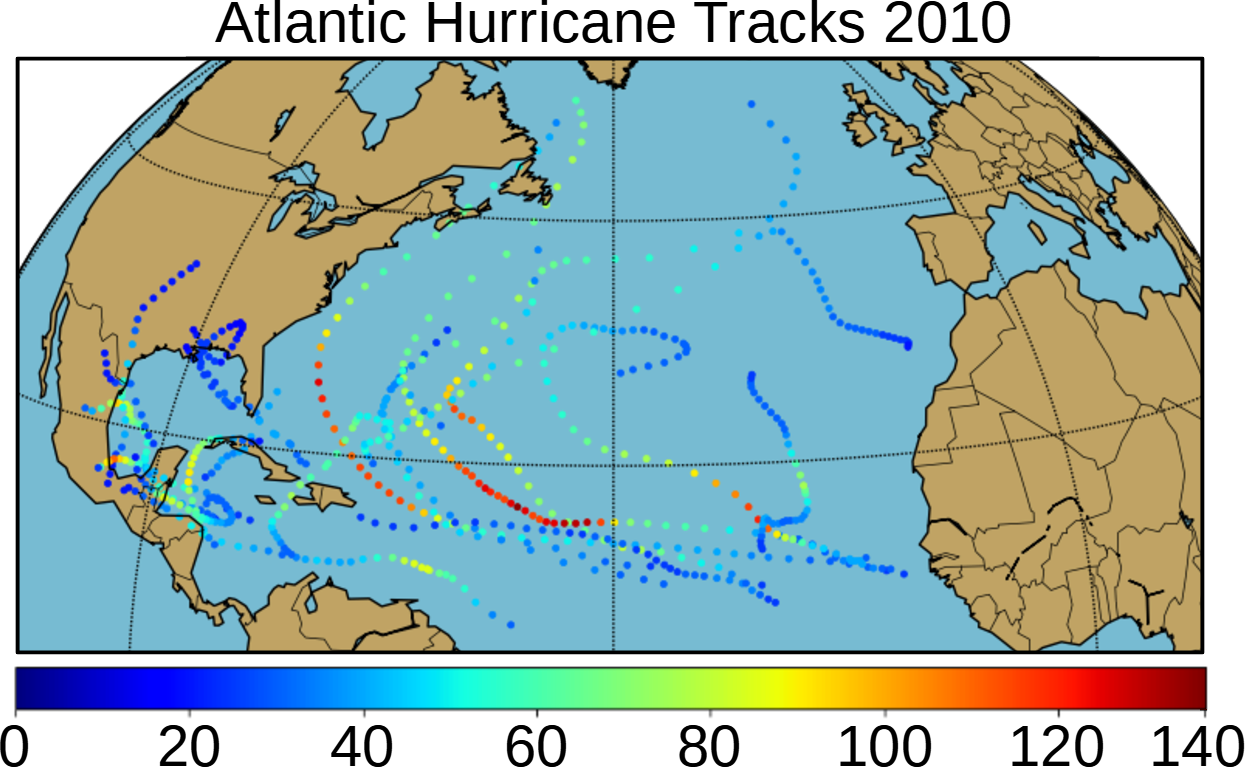}
	\caption{Maximum sustained wind (in knots): Color-coded for 2010 hurricane tracks (right) and histogram of maxima for all hurricanes (left).}
	\label{fig:hurricanes}
\end{figure}

\subsection{Spline Model}

Before turning towards group-level analysis, we first derive spline representations for the trajectories.
To this end, we consider the tracks to be observations of a manifold-valued random variable that depends on an explanatory scalar variable.
Then the relationship of the hurricane location on (elapsed) time is modeled as a cubic spline $B$.
The estimation of model parameters---the control points of the spline in our setting---is known as regression problem for which Riemannian generalizations are readily available.
In particular, we employ the least-squares-based approach presented in~\cite{Hanik_ea2020} that can be shown to provide maximum likelihood estimators for $S^2$-valued trajectories.

While we restrict our attention to cubic splines, we need to select the number of segments in order to obtain a fixed parametric model.
We can adopt a qualitative selection strategy based on the goodness of fit as determined by the \textit{coefficient of determination} denoted $R^2$.
Specifically, let $(q_1, t_1), \ldots, (q_N, t_N) \in U \times [0, L]$ be the observations of a track at corresponding (normalized) times, then the geometric $R^2$-value~\cite{Fletcher2013} is given by
\begin{equation}
    R^2 = 1 - \frac{\text{unexplained variance}}{\text{total variance}} = 1 - \frac{\frac{1}{N}\! \sum_i d(B(t_i), q_i)^2}{\min_{p \in U} \frac{1}{N}\! \sum_i d(p, q_i)^2},
\end{equation}
where the variances are determined by sums of squared geodesic distances $d$ of the data to the model $B$ respectively a single, best-fitting point in $U$. The latter is also referred to as Fr{\'e}chet mean.
As both variances are nonnegative, we have that $R^2\leq1$ with equality if and only if the model $B$ perfectly fits the data.

With the $R^2$ statistic at hand, we perform regression analysis for the hurricane tracks employing single- as well as two-segment splines.
Fig.~\ref{fig:regression} shows a histogram of $R^2$ values for the estimated splines together with visualizations for two exemplary chosen hurricanes.
While both spline representations expose a very high fidelity overall, the shift in distribution over $R^2$ values toward 1 confirms a significant improvement of the two-segment model over single-segment ones.
Given the upper bound of 1 for the $R^2$ value, there is minor room for improvement by more complex models with three or more segments, thus rendering them inadequate according to the principle of parsimony (also known as Ockham's razor).


\begin{figure}[t]
  \centering
  \includegraphics[width=0.43\textwidth]{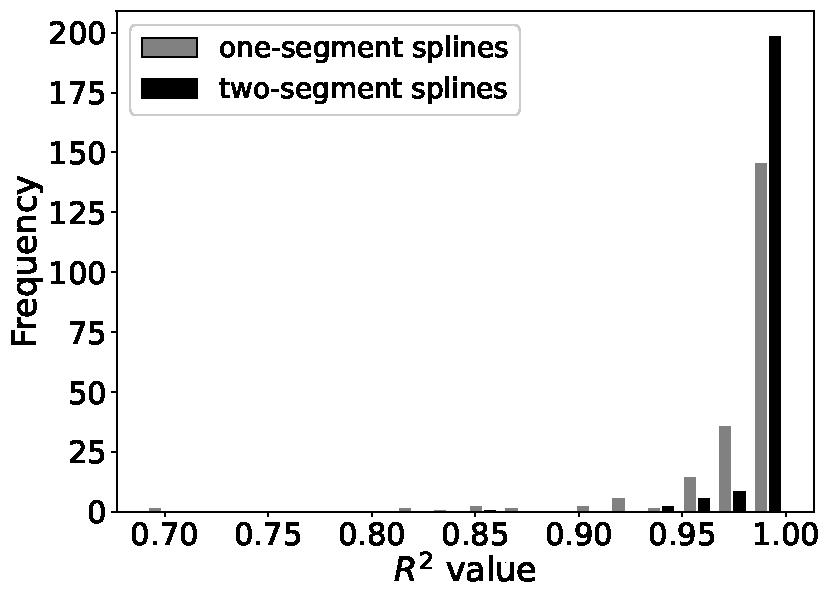} 
  {%
    \setlength{\fboxsep}{0pt}%
    \setlength{\fboxrule}{.5pt}%
    \raisebox{0.58cm}{\fbox{\includegraphics[width=0.52\textwidth]{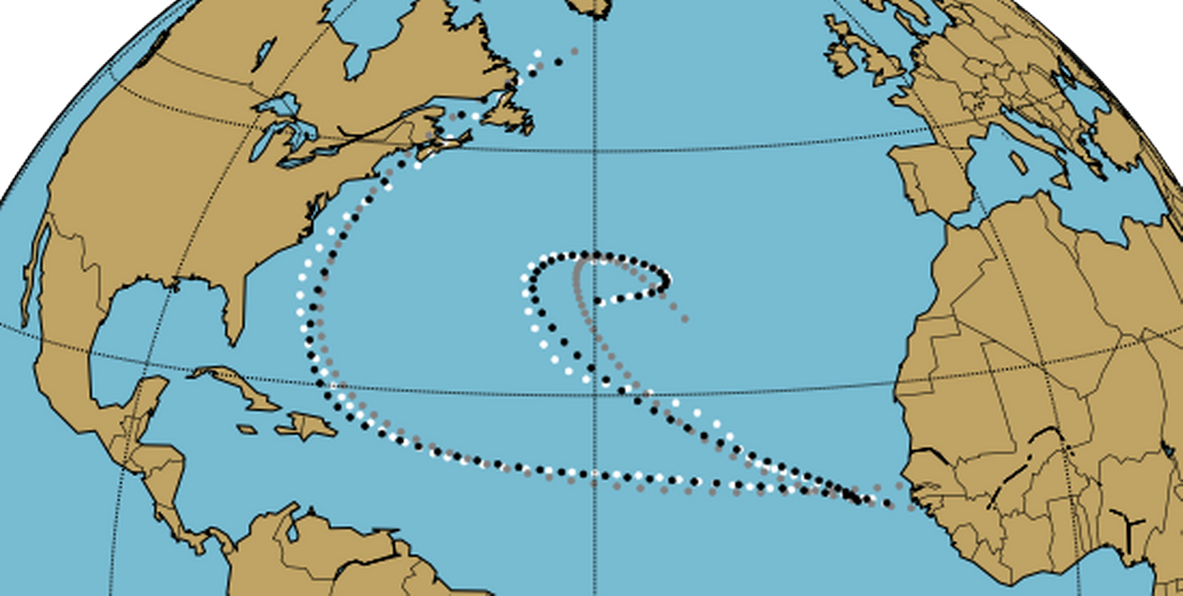}}}%
   }
  \caption{Left: Histogram of $R^2$ values for the regressed cubic Bézier curves. Right: Two exemplary hurricane tracks (white) together with regressed one- and two-segment curves (gray/black) with $R^2$ values of $0.995$/$0.998$ and $0.916$/$0.993$, respectively.}
  \label{fig:regression}
\end{figure}


\subsection{Group-level Analysis}


\begin{figure}[t]
    \centering
    {%
    \setlength{\fboxsep}{0pt}%
    \setlength{\fboxrule}{.5pt}%
    \fbox{\includegraphics[width=0.47\textwidth]{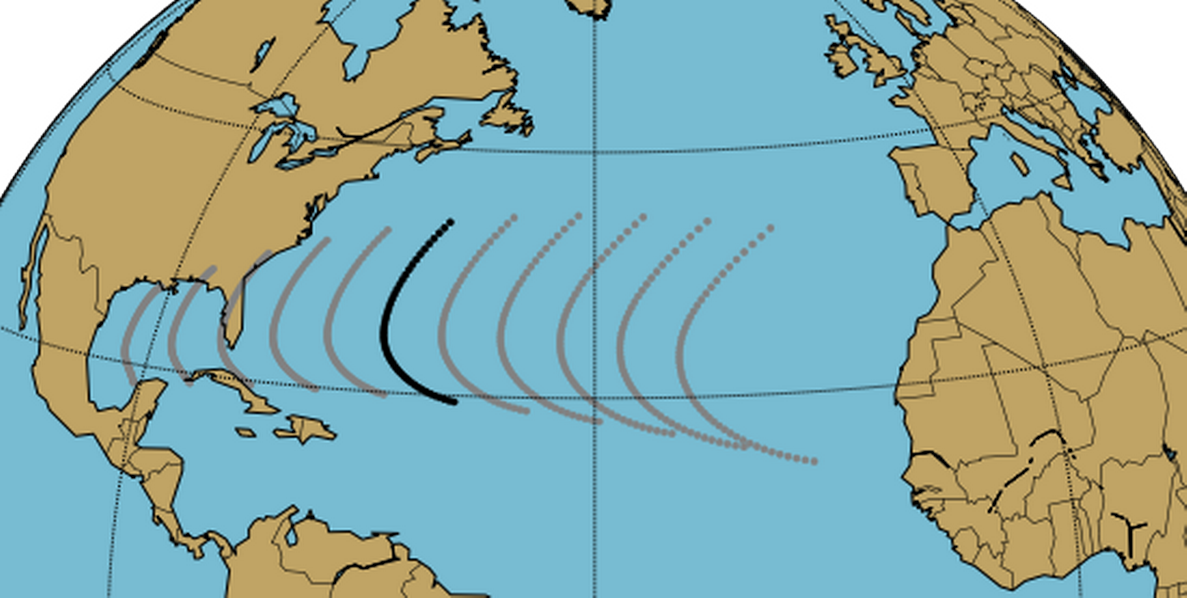}}%
    }%
    \hspace{1ex}
    {%
    \setlength{\fboxsep}{0.2pt}%
    \setlength{\fboxrule}{.5pt}%
    \fbox{\includegraphics[width=0.47\textwidth]{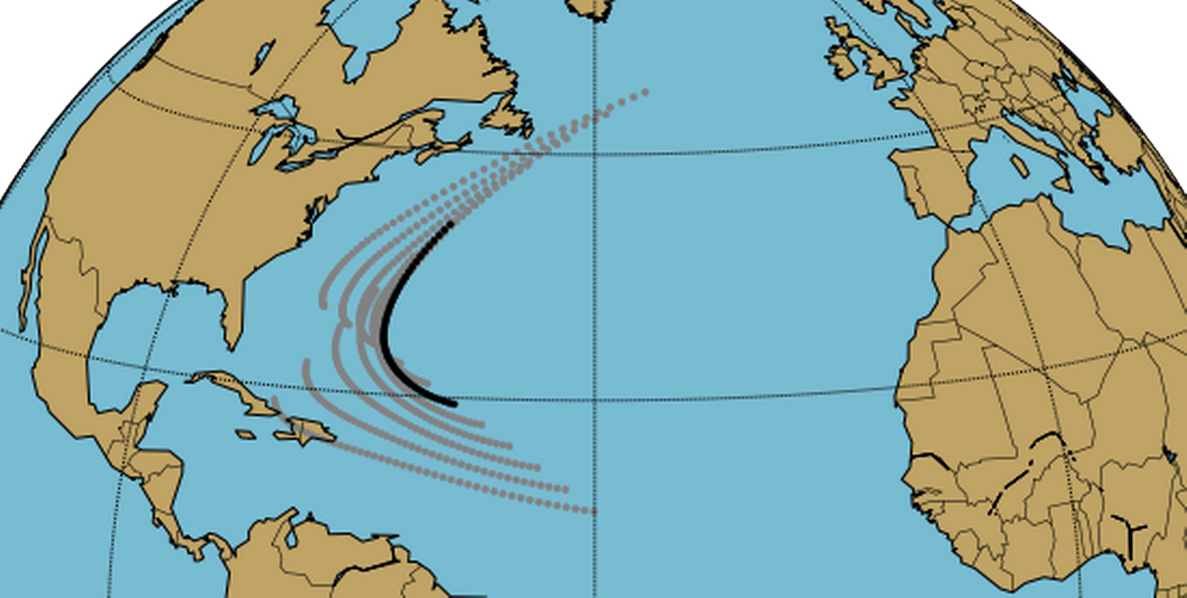}}%
    }%
    \caption{Two-segment spline Fr\'{e}chet mean (black) and first two dominant modes of variation (gray; left: first mode, right: second mode) for hurricane tracks.}
    \label{fig:pga_modes}
\end{figure}

\begin{figure}[b]
  \centering
  \includegraphics[width=.495\textwidth]{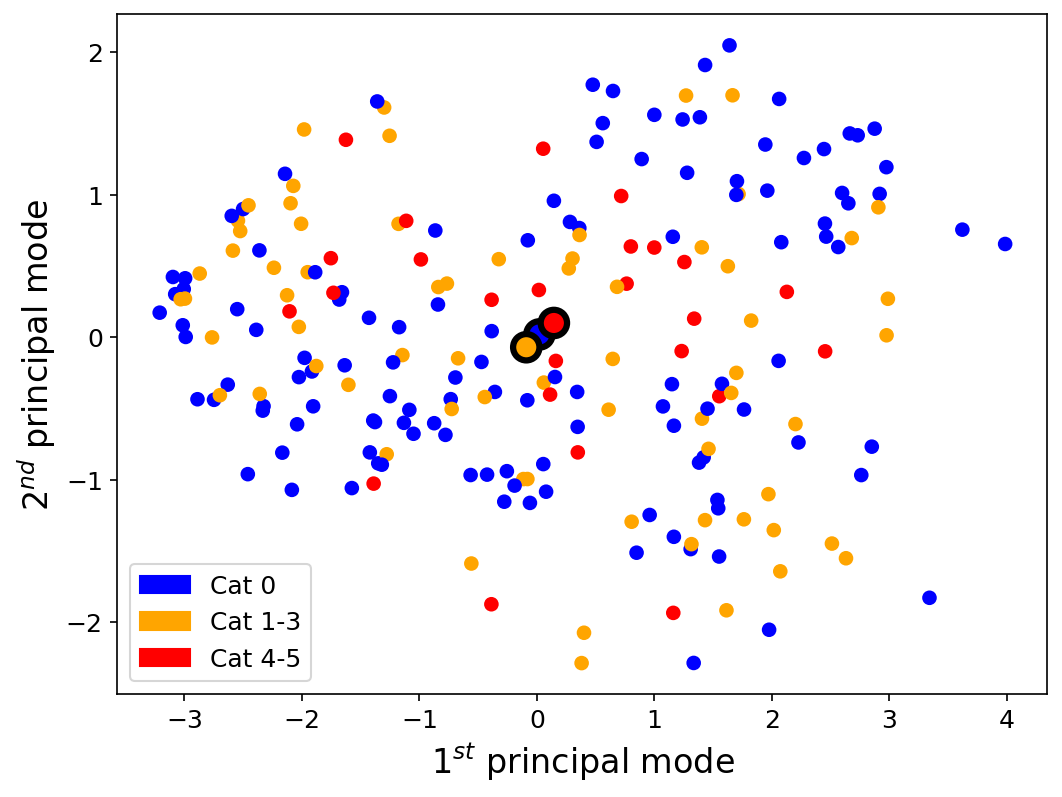}
  \includegraphics[width=.48\textwidth,trim=27 0 0 0,clip]{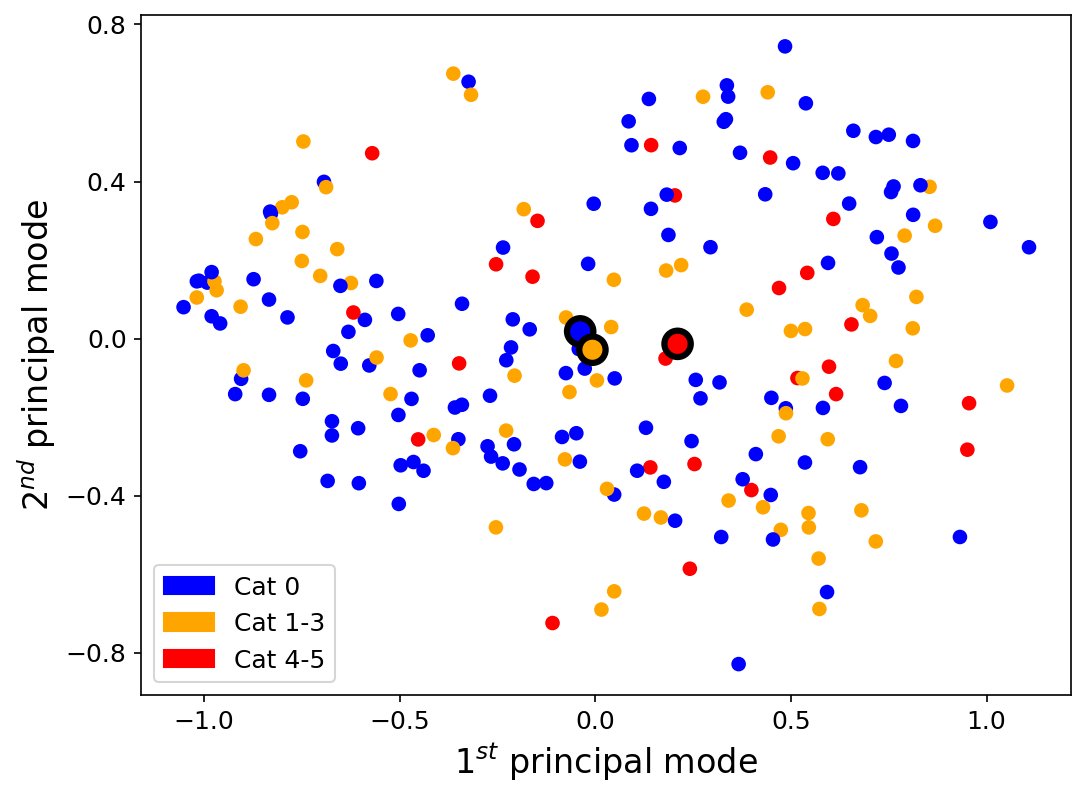}
  \vspace{-1em}
  \caption{PGA loading plot of the first two principal geodesic modes for proposed (right) and $L^2$ (left) metric. Group-wise means are highlighted with black boundaries.}
  \label{fig:scatter_plot}
\end{figure}

The presented geometric framework provides means to investigate hurricane tracks on a group level by studying their interrelations in terms of similarities and characteristic variations.
In particular, the proposed metric induces a notion of variance and co-variances allowing to perform mean-variance analysis.
To this end, we perform principal geodesic analysis (PGA) as proposed in~\cite{fletcher2004pga}.
The estimated Fr\'{e}chet mean together with the first and second dominant geodesic modes are visualized in Fig.~\ref{fig:pga_modes}. 

We can further encode hurricane tracks with respect to the hierarchical basis determined by PGA, i.e.\ the dominant geodesic modes.
By further omitting coordinates corresponding to the least dominant modes we obtain a low-dimensional representation that lends itself to visualization of the data.
Henceforth, let $\M$ denote the set of absolutely continuous functions $[0,\,1]\to M$. Physically motivated, we consider three groups: (i) tropical storms/depressions (category $<1$), (ii) hurricanes (category $1-3$) with some to devastating damage, and (iii) major hurricanes (category $>3$) with catastrophic damage. Fig.~\ref{fig:scatter_plot} shows a scatter plot based on a two-dimensional encoding with respect to the proposed representation in B\'{e}zierfold $\mathcal{B}$ in juxtaposition to treatment as elements of $\M$ equipped with the $L^2$ metric. In contrast to the latter, our representation shows an increased class separation particularly evident by the higher disparity between the group-wise means.

\subsection{Intensity Classification}

In the following, we investigate to which extent the (maximal) intensity of a hurricane can be inferred from its trajectory.
Note that, while the dependency of maximum sustained wind on physical parameters and its spatiotemporal correlation is highly complex, in this experiment we are primarily interested in the discrimination ability of features derived from our proposed framework in comparison to those obtained from state-of-the-art.
In particular, we will employ the hierarchical encoding provided by PGA (see also previous section) as features for classification.
Due to the limited amount of available hurricane tracks and the strong disbalance with respect to hurricane categories, we opt for a support vector machine (SVM) as classifier.
Thereby we condition the SVM model to differentiate between the three intensity classes described above. For all experiments we employ the implementation available in the \texttt{scikit-learn} v1.0 library using radial basis functions (kernel coefficient \texttt{gamma}=0.7) and regularization parameter \texttt{C}=3.
We further perform balanced training by adapting class weights through a factor inversely proportional to class frequencies.

As a baseline, we compare our approach to the common $L^2$ metric on $\M$. Thereby, to obtain a discrete counterpart $\M^h$ of $\M$ all hurricane tracks have been re-sampled to 32 points equidistantly spread (on the normalized time interval $[0,1]$).
Additionally, we compare to two state-of-the-art approaches. On the one hand, we employ the square root velocity (SRV) framework from~\cite{bauerHomog} that provides an elastic metric on $\M$.
On the other hand, we compare to the functionally-based metric for the intrinsic splines proposed in~\cite{HanikHegevonTycowicz2022}, i.e.\ the $L^2$ metric on $\mathcal{B}$.

\begin{figure}[b!]
    \centering
    \includegraphics[width=\textwidth]{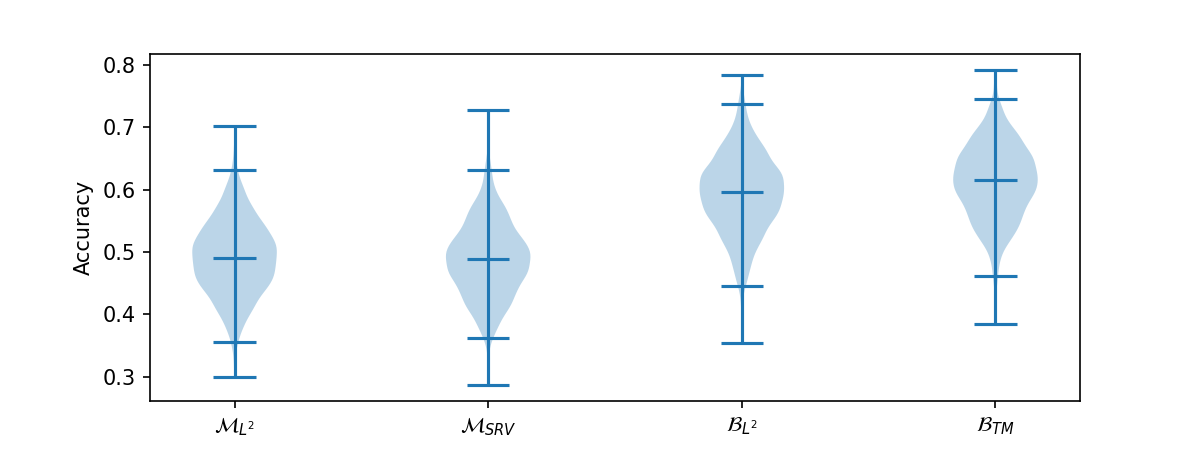}
    \vspace{-4ex}
    \caption{Distribution of classification accuracy for hurricanes treated as (from left to right) immersed curves in~$\M$ equipped the $L^2$ and SRV-based elastic metric as well as splines in $\B$ equipped with the $L^2$ and the proposed metric.}
    \label{fig:accuracy}
\end{figure}

For the evaluation of classification performance, we employ a balanced accuracy score given as the average of recall obtained in each class.
Due to the limited cardinality of the hurricane data set, the partition into training and test data can emphasize or dampen the extent of problems like overfitting and selection bias.
Therefore, we follow an extensive validation strategy in terms of 1000 random repetitions of 3-fold, stratified cross-validation.
This allows for estimating the distribution of classification accuracies, which are shown in Fig.~\ref{fig:accuracy} for all methods under consideration.
The results reveal significantly improved discrimination ability of spline-based representation with $61\%$ and $59\%$ accuracy for our proposed and the functional-based metric on average, respectively, as compared to $\M$-based ones with both $\approx 49\%$ accuracy.
These findings quantitatively confirm the superior performance for differentiation that was already qualitatively visible in the low-dimensional plots in Fig.~\ref{fig:scatter_plot}.
Note that the $dim(\M^h)=64$ is considerably larger than $dim(\mathcal{B})=8$ leading to SVM with more degrees of freedom in the former case.
Nevertheless, we would like to emphasize that the discrepancy in classification performance prevails also for dimensionality-reduced PGA encoding of curves in $\M$. This suggests that the increased discrimination performance can be attributed to the ability of regression schemes to suppress confounding factors such as variances in parameterization or noise.

\subsection{Computational Performance}

We conclude the experiments with an evaluation of the computational performance of algorithmic schemes for the analysis of manifold-valued splines.
In particular, we compare our approach with the functional-based metric~\cite{HanikHegevonTycowicz2022} for the B\'ezierfold.
Neither of these metrics provides explicit expressions for geodesics in $\mathcal{B}$.
We, therefore, resort to iterative optimization based on variational time-discrete geodesics~\cite{rumpf2015variational}.
For the functional-based metric, we use the implementation by~\citet{HanikHegevonTycowicz2022} available via the \texttt{morphomatics}~\cite{Morphomatics} v2.1.1 library.
Based on this library, we also implemented time-discrete Sasaki geodesics as described in~\cite{MuFl2012,NavaYazdani2022sasaki}.
For both metrics, we employ discrete 10-geodesics, i.e.\ intrinsic polygonal paths consisting of 10 segments.
Furthermore, the functional-based metric is evaluated using a 4-point quadrature scheme.

In terms of complexity the mean-variance analysis, namely~PGA, is the most costly part of our experiments as it requires iterative schemes with geodesic evaluations in each iteration. 
On a desktop computer (Intel i9-10920X CPU, NVIDIA GeForce RTX 3090 GPU), we obtain run times for PGA computation of 28s and 164s for the proposed and the functional-based metric, respectively.
This considerable gap in performance can be attributed to the fact that the functional-based metric requires additional discretization of integrals along the trajectories involving geodesic distance on the base manifold---$S^2$ in our experiments.
Furthermore, since explicit expressions for curvature estimation on $S^2$ (and indeed many other practically relevant manifolds) are known, the Sasaki metric can be approximated very efficiently.

\section{Conclusion and Future Work}

We presented a Riemannian framework for the comparison of manifold-valued trajectories. Therein, trajectories are represented by composite B\'ezier splines gained via regression. For the comparison of the trajectories, we proposed a natural extension of the Sasaki metric. This allows for estimating average trajectories representing group trends. We applied the proposed model to Atlantic hurricane tracks and presented experiments for an intensity classification of the tracks. Therein, we compared the proposed framework with those based on $L^2$ as well as elastic metric, in which the results indicate clear  advantages of our approach.

For future work, we intend to propose a generative model based on the introduced framework to enable online forecasting, employ statistical tests, and incorporate the main parameter determining the intensity of hurricanes, namely\ the Coriolis force and thermal effects, to improve the modeling. Moreover, we plan to apply our approach to further manifold-valued data and employ a geometry-aware recurrent neural network for classification.

\section{Acknowledgments}
We are grateful for the HURDAT 2 database provided by the U.S. National Oceanic and Atmospheric Administration.
This work was supported through the German Research Foundation (DFG) via individual funding (project ID 499571814) as well as under Germany´s Excellence Strategy – MATH+ : The Berlin Mathematics Research Center (EXC-2046/1 – project ID: 390685689), by the Bundesministerium f\" ur Bildung und Forschung (BMBF) through BIFOLD -- The Berlin Institute for the Foundations of Learning and Data (ref. 01IS18025A and ref. 01IS18037A), and by the Bundesministerium f\"ur Wirtschaft und Klimaschutz through DAKI-FWS (01MK21009J).

\bibliography{main}

\end{document}